\documentclass[11pt]{amsart}
\usepackage{amssymb,amstext,amsmath,amscd,amsthm,amsfonts,enumerate,graphicx,latexsym}
\usepackage{extarrows}
\usepackage{stackrel}
\usepackage[all]{xy}
\usepackage[usenames]{color}

\newtheorem{theorem}{Theorem}[section]
\newtheorem{lemma}[theorem]{Lemma}
\newtheorem{corollary}[theorem]{Corollary}
\newtheorem{proposition}[theorem]{Proposition}
\theoremstyle{definition}
\newtheorem{definition}[theorem]{Definition}

\newtheorem{remark}[theorem]{Remark}
\numberwithin{equation}{section}
 
\title[Krull--Gabriel dimension and Cohen-Macaulay representations]{Krull--Gabriel dimension and Cohen-Macaulay representations}
\subjclass[2020]{Primary 13C14;  Secondary 16G60.}
\date{\today}
\thanks{NH was partly supported by JSPS KAKENHI Grant Number 25K06966. 
YY was partly supported by JSPS KAKENHI Grant Number 24K0669.}
\author{Naoya Hiramatsu} 
\address{Institute for the Advancement of Higher Education, Okayama University of Science, 1-1 Ridai-cho, Kita-ku, Okayama 700-0005, Japan}
\email{n-hiramatsu@ous.ac.jp} 
\author{Yuji Yoshino} 
\address{Graduate School of Environmental, Life, Natural Science and Technology, Okayama University, Okayama 700-8530, Japan}
\email{yoshino@math.okayama-u.ac.jp} 

\begin{document}
\def\AA{\mathcal A}
\def\add{\mathrm{add}} 
\def\CC{\mathcal C}
\def\CML{{\mathcal C}(\Lambda )} 
\def\CMR{{\mathcal C}(R)} 
\def\Coker{\mathrm{Coker}\,}
\def\colim{\operatorname*{colim}}
\def\End{\mathrm{End}}
\def\Ext{\mathrm{Ext}}
\def\Hom{\mathrm{Hom}}
\def\HomA{\mathrm{Hom}_{\AA}}
\def\HomqA{\mathrm{Hom}_{\AA / \AA_0}}
\def\HomAS{\mathrm{Hom}_{\AA / \SS}}
\def\HomR{\mathrm{Hom}_R}
\def\Im{\mathrm{Im}\,}
\def\Ind{\mathrm{Ind}}
\def\Ker{\mathrm{Ker}\,}
\def\KGdim{\mathrm{KGdim}\,}
\def\m{\mathfrak m}
\def\Mod{\mathrm{Mod}\,}
\def\mod{\mathrm{mod}\,}
\def\ModCC{{\mathrm{Mod}}\,\CC}
\def\modCC{{\mathrm{mod}}\,\CC}
\def\ModsCC{{\mathrm{Mod}}\,\sCC}
\def\modsCC{{\mathrm{mod}}\,\sCC}
\def\ModCML{{\mathrm{Mod}}\,\CML}
\def\modCML{{\mathrm{mod}}\,\CML}
\def\ModsCML{{\mathrm{Mod}}\,\sCML}
\def\modsCML{{\mathrm{mod}}\,\sCML}
\def\modR{{\mathrm{mod}}\,R}
\def\p{\mathfrak p}
\def\q{\mathfrak q}
\def\rad{\mathrm{rad}\,}
\def\SS{\mathcal S} 
\def\Spec{\mathrm{Spec}\,} 
\def\sCC{\underline{\mathcal C}}
\def\sCML{\underline{{\mathcal C}}(\Lambda )}  
\def\sCMR{\underline{{\mathcal C}}(R)}  
\def\sEnd{\underline{\mathrm{End}}}
\def\sHom{\underline{\mathrm{Hom}}}
\def\sHomR{\underline{\mathrm{Hom}}_R}
\def\smod{\underline{\mathrm{mod}}}
\def\soc{{\mathrm{soc}}\,}
\def\Supp{{\mathrm{Supp}}\,}
\def\syz{\mathrm{syz}\,}
\def\top{{\mathrm{top}}\,}
\def\Z{\mathbb Z}

\def\gldim{\mathrm{gl.dim}}
\def\HomL{\mathrm{Hom}_\Lambda}
\def\sHomL{\underline{\mathrm{Hom}}_\Lambda}
\def\LL{\Lambda}
\def\modL{{\mathrm{mod}}\,\Lambda}
\def\PP{\mathcal P}
\def\II{\mathcal I}
\begin{abstract}
Let $R$ be a complete Cohen--Macaulay local ring and $\LL$ an $R$-order. 
We study the Krull--Gabriel dimension of the functor category $\modsCML$, where $\sCML$ is the stable category of maximal Cohen--Macaulay $\LL$-modules. 
This work is motivated by the non-existence theorem of Herzog and Krause for Krull--Gabriel dimension $1$ over artin algebras.  We first prove that, if $\LL$ is Gorenstein and $\KGdim \modsCML \leq 1$, then $\LL$ is an isolated singularity.
We then show that, if $\LL$ is an isolated singularity of uncountable Cohen--Macaulay representation type, then $\KGdim \modsCML \neq 1$.
\end{abstract}
\maketitle
\section{Introduction}\label{Intro}
The Krull--Gabriel dimension of an abelian category is an invariant that measures the complexity of the category in terms of a filtration by Serre subcategories. 
It was introduced by Gabriel \cite{Ga62} and it has played an important role in representation theory; see, for example,  Geigle \cite{G85} and Schr\"oer \cite{S00}. 

Let $A$ be a finite dimensional algebra over a field $k$ and let $\mod A$ denote the category of finitely generated $A$-modules. 
The category $\mod \mod A$ of finitely presented functors on $\mod A$ is abelian, and its Krull--Gabriel dimension, denoted by $\KGdim \mod \mod A$, is closely related to the representation type of $A$. 
Auslander's theorem \cite{A80} states that $A$ is of finite representation type if and only if $\KGdim \mod \mod A = 0$. 
Geigle \cite{G85} showed that tame hereditary algebras have Krull--Gabriel dimension $2$, while algebras of wild representation type have infinite Krull--Gabriel dimension \cite{Kr01}.

One of the striking results in this direction is the non-existence of Krull--Gabriel dimension $1$.
For artin algebras, it was obtained independently in the following form by Herzog and Krause.

\begin{theorem}[{\cite[Theorem 3.6]{He97}, \cite[Corollary 11.4]{Kr98}}]\label{Intro1}
Let $A$ be a finite dimensional algebra over an uncountable algebraically closed field. 
Then $\KGdim \mod \mod A \neq 1$. 
\end{theorem}

The purpose of this paper is to prove an analogous non-existence theorem in Cohen--Macaulay representation theory. 
Let $R$ be a complete Cohen--Macaulay local ring and $\LL$ an $R$-order. 
Let $\CML$ denote the category of maximal Cohen--Macaulay $\LL$-modules. 
We denote by $\sCML$ the stable category of $\CML$, and consider the category $\modsCML$ of finitely presented functors on $\sCML$. 
The Krull--Gabriel dimension of $\modsCML$ measures the complexity of the stable category of maximal Cohen--Macaulay modules and is therefore naturally related to the Cohen--Macaulay representation type of $\LL$.

Our first main result shows that low Krull--Gabriel dimension imposes a strong singularity theoretic restriction.
Namely, the first two stages of the Krull--Gabriel filtration already force the order $\LL$ to be an isolated singularity.

\begin{theorem}[Theorem \ref{B10}]\label{Intro2}
Suppose that $\LL$ is Gorenstein and that $\KGdim \modsCML \le 1$. 
Then $\LL$ is an isolated singularity. 
\end{theorem}

Thus, for a Gorenstein order, Krull--Gabriel dimension at most one can occur only in the isolated singularity case.  
This gives a connection between the low dimensional part of the Krull--Gabriel filtration and the singular locus of the base ring.

To prove the non-existence result for Krull--Gabriel dimension $1$, Herzog and Krause used the theory of generic modules due to Crawley-Boevey \cite{CB92}. 
Without constructing generic Cohen--Macaulay modules, we obtain the following non-existence theorem for Krull--Gabriel dimension $1$ in the case of uncountable CM representation type. 

\begin{theorem}[Theorem \ref{C1}]\label{Intro3}
Let $\LL$ be an isolated singularity but not necessarily Gorenstein. 
Suppose that $\LL$ is of uncountable CM representation type. 
Then $\KGdim \modsCML > 1$. 
\end{theorem}

Our approach provides an alternative to the generic module method of Herzog and Krause. 
Rather than passing to generic or pure-injective objects in the ambient functor category, our argument is driven by the structure of finitely presented functors on the stable Cohen--Macaulay category.
The key idea is to analyze their socle and radical series and to make essential use of the direct limits of the socle series and the inverse limit of the radical series.
This enables us to control the supports of simple objects in the quotient category without introducing generic Cohen--Macaulay modules. 

Even if $\LL$ is Gorenstein and an isolated singularity, it is not clear in general whether countably infinite CM representation type can occur. 
In contrast, over an uncountable algebraically closed field, a finite dimensional algebra is either of finite representation type or  of uncountable representation type by the second Brauer--Thrall theorem. 
Combining Theorem \ref{Intro2} and Theorem \ref{Intro3}, if $\LL$ is a Gorenstein order, Krull--Gabriel dimension $1$ forces that $\LL$ is an isolated singularity and of countably infinite CM representation type. 
In this sense, the Krull--Gabriel dimension $1$ case provides a new perspective on the well-known countable Cohen--Macaulay representation type problem.

As an application, we also obtain a stable analogue of the non-existence theorem for finite dimensional algebras.

\begin{corollary}\label{C11_intro}\label{Intro4}
Let $A$ be a finite dimensional algebra over an uncountable algebraically closed field.
Then $\KGdim \mod \underline{\mod A} \neq 1$. 
\end{corollary}

Here $\underline{\mod A}$ denotes the stable category of $\mod A$ modulo projective modules.

The paper is organized as follows.
In Section \ref{A}, we recall basic facts on quotient categories and realize a quotient category of $\modsCC$ as a subcategory of $\ModsCC$ (Theorem \ref{A8}).
In Section \ref{B}, we investigate the endomorphism rings of simple objects in $\modsCML/(\modsCML)_0$ and prove Theorem \ref{B10}.
Section \ref{C} is devoted to the proof of Theorem \ref{C1}.

\section{Preliminaries}\label{A}

For an abelian category, all subcategories are assumed to be full. 
Let $\AA$ be an abelian category. 
We say that a subcategory $\SS$ of $\AA$ is a Serre subcategory if $\SS$ is closed under taking subobjects, quotient objects, and extensions. 

\begin{definition}\label{A1}
We say that a morphism $f: X \to Y$ in $\AA$ is a pseudo-monomorphism (resp. a pseudo-epimorphism) with respect to a Serre subcategory $\SS$ if $\Ker f$ (resp. $\Coker f$) belongs to $\SS$. 
We also say that a morphism $f: X \to Y$ in $\AA$ is a pseudo-isomorphism if $f$ is a pseudo-monomorphism and a pseudo-epimorphism, that is, $\Ker f$ and $\Coker f$ belong to $\SS$.
\end{definition}

We note that the set $\Sigma$ of pseudo-isomorphisms forms a multiplicative system in $\AA$ (cf. \cite[2.1]{SY18}). 
Thus, we define the quotient category $\AA / \SS$ by the localization $\Sigma^{-1} \AA$. 
That is, the objects are the same as $\AA$ and the morphism set in $\AA / \SS$ is given by 
$$
\HomAS (\overline{X}, \overline{Y}) = \{ f a ^{-1} \mid a \in \Sigma \} / \sim .
$$
Here $f a ^{-1}$ denotes the diagram $X \xleftarrow{a} X' \xrightarrow{f} Y$ for some $X' \in \AA$, and the equivalence relation $\sim$ is generated by $f a^{-1} \sim (fb)(ab)^{-1}$ where $b: X'' \to X'$ is a morphism in $\Sigma$. 

Note that $\AA/ \SS$ is an abelian category and there is a natural exact functor $\AA \to \AA /\SS$. 
It is well known that there are isomorphisms (see \cite[Remark 7.1.17(ii)]{KS}):
$$
\HomAS (\overline{X}, \overline{Y}) \cong \colim_{{P_X}^{op} \ni (X' \to X)} \Hom_{\AA} (X', Y)\cong \colim_{ P_Y \ni (Y \to Y')} \Hom_{\AA} (X, Y'),  
$$ 
where, for $X \in \AA$, $P_X$ denotes the category whose objects are pseudo-isomorphisms $f:X\longrightarrow X'$ in $\AA$, and whose morphisms from $f : X\to X'$ to $g: X \to X''$ are morphisms $h: X' \to X''$ such that $h\circ f=g$.
Since pseudo-isomorphisms form a multiplicative system, $P_X$ is a filtered category.

\begin{remark}\label{A3}
By the universality of the quotient category, one can show that $f a^{-1}: \overline{X} \to \overline{Y}$ is a monomorphism in $\AA / \SS$ if and only if $f$ is a pseudo-monomorphism. 
Similarly, $f a^{-1}: \overline{X} \to \overline{Y}$ is an epimorphism in $\AA / \SS$ if and only if $f$ is a pseudo-epimorphism. 
\end{remark}

We consider the subcategories $\SS ^{\perp} = \{ X \in \AA \mid  \HomA (S, X) = 0 \ \text{for all } S \in \SS \}$ and ${}^{\perp}\!\SS  = \{ X \in \AA \mid  \HomA (X, S) = 0 \ \text{for all } S \in \SS \}$.

\begin{proposition}\label{A2}
\begin{enumerate}[\rm(1)]
\item Let $X \in {}^{\perp}\! \SS$ and $Y \in \SS ^{\perp}$. Then $\HomAS (\overline{X}, \overline{Y}) \cong \HomA (X, Y)$. 
\item For an exact sequence $0 \to \overline{X} \to \overline{Y} \to \overline{Z} \to 0$ in $\AA / \SS$, there exist $X' , Z' \in \AA$ such that there exists an exact sequence $0 \to X' \to Y \to Z' \to 0$ in $\AA$ with $\overline{X'} \cong \overline{X}$, $\overline{Z'} \cong \overline{Z}$ in $\AA/\SS$. 
\end{enumerate}
\end{proposition}

\begin{proof}
(1) Let $f a^{-1} \in \HomAS (\overline{X}, \overline{Y})$. 
Then we have the diagram: $X \xleftarrow{a} X' \xrightarrow{f} Y$. 
Since $a$ is a pseudo-isomorphism, both $\Ker a$ and $\Coker a$ belong to $\SS$. 
Therefore $a$ is an epimorphism since $X \in {}^{\perp}\! \SS$.
Furthermore, since the image $f(\Ker a)$ belongs to $\SS$ and $Y \in \SS^{\perp}$, it follows that $f(\Ker a) = 0$. 
Consequently, $f$ factors through $a : X' \to X$. 
Hence, the natural map from $\HomA (X, Y)$ to $\HomAS (\overline{X}, \overline{Y})$ is surjective. 
The injectivity of this map follows by an argument similar to the one above.

(2) Suppose that $f a^{-1}: \overline{X} \to \overline{Y}$, represented by $X \xleftarrow{a} X'' \xrightarrow{f} Y$ is a monomorphism in $\AA/ \SS$. 
Set $X' = \Im f$. 
Then $X'$ is a subobject of $Y$, and the natural epimorphism $\pi : X'' \to X'$ is a pseudo-isomorphism since $\Ker f$ belongs to $\SS$. 
Set $Z' = Y/ X'$. 
Note that $\overline{f} =\overline{i \circ \pi}= \overline{i} \circ \overline{\pi}$ in $\AA/\SS$. 
Since $\overline{\pi}$ is an isomorphism in $\AA/\SS$, we have $\overline{Z} \cong \Coker \overline{f} \cong \Coker ( \overline{i} \circ \overline{\pi}) \cong \Coker \overline{i} \cong \overline{Z'}$. 
Also we have $\overline{X} \cong \overline{X'}$. 
\end{proof}

The following corollary is useful for showing that an object is simple in a quotient category. 

\begin{corollary}\label{A4}
Let $\AA$ be an abelian category and $\SS$ a Serre subcategory. 
An object $Y$ of $\AA \setminus \SS$ is simple in $\AA/ \SS$ if and only if, for every short exact sequence $0 \to X' \to Y \to Z' \to 0$, either $X'$ or $Z'$ belongs to $\SS$. 
\end{corollary}

\begin{proof}
The statement follows from Proposition \ref{A2}(2). 
See also \cite[Lemma 1.1]{GR74}. 
\end{proof}

\begin{proposition}\label{A5}
Suppose that the inclusion functor $\SS \to \AA$ admits a right adjoint functor $\gamma$. 
For every $\overline{X} \in \AA/ \SS$, there exists $X' \in \SS ^{\perp}$ such that $\overline{X} \cong \overline{X'}$ in $\AA/\SS$. 
\end{proposition}

\begin{proof}
Let $X \in \AA$. 
Then we have a short exact sequence $0 \to \gamma (X) \to X \to X' \to 0$ where $X' = X / \gamma (X)$. 
By the standard argument of torsion theory, $\gamma (X)$ is the maximal subobject of $X$ that belongs to $\SS$, and thus $\gamma (X') = 0$ (e.g. \cite[Lemma 2.1]{Kr97}). 
Hence $\HomA (S, X') \cong \HomA (S, \gamma (X')) = 0$ for each $S \in \SS$. 
Therefore $X' \in \SS ^{\perp}$. 
It is clear that $\overline{X} \cong \overline{X'}$ in $\AA/ \SS$. 
\end{proof}

In the rest of the paper, we are interested in a category of (finitely presented) functors which is abelian. 
For an additive category $\CC$, we denote by $\ModCC$ the category of {\it contravariant} additive functors from $\CC$ to the category of abelian groups, with natural transformations as morphisms. 
We also denote by $\modCC$ the full subcategory of $\ModCC$ consisting of all finitely presented functors. 
$$
\modCC  = \{ F : \CC \to \mathrm{Ab} \mid \Hom _\CC (\  , N) \to \Hom _\CC (\  , M) \to F \to 0 \ \text{with $M, N \in \CC$} \}.
$$
A skeletally small additive category $\CC$ is called {\it coherent} when $\modCC$ is abelian. 

Let $\CC$ be coherent and set $\AA = \modCC$. 
Let $\SS$ be a Serre subcategory of $\AA$. 
At the end of this section, we realize $\AA/\SS$ as a subcategory of $\ModCC$. 

Since $\ModCC$ admits all direct limits, we make the following definition.   

\begin{definition}\label{A9}
Let $F \in \AA$. 
We define $\displaystyle \widetilde{F} = \colim_{P_F \ni (F \to F')} F'$. 
We also denote by $\widetilde{\AA }$ the full subcategory of $\ModCC$ consisting of functors $\widetilde{F}$ for $F \in \AA$. 
\end{definition}

\begin{lemma}\label{A6}
Let $F$, $G \in \AA$. 
There is a natural isomorphism: $$\Hom_{\AA /  \SS } ( \overline{G}, \overline{F}) \cong \Hom_{\ModCC} (G, \widetilde{F}).$$ 
\end{lemma}

\begin{proof}
Suppose that $G \cong \Hom _{\CC}(\  , M)$ for some $M \in \CC$. 
Then we have
$$
\begin{array}{ll}
\Hom_{\AA /  \SS } ( \overline{G}, \overline{F}) &\cong \colim_{P_F} \Hom_{\AA} (\Hom _{\CC} (\  , M), F' )\\
&\cong \colim_{P_F}  F'(M) = \widetilde{F} (M) \cong \Hom_{\ModCC} (\Hom _{\CC}(\  , M), \widetilde{F}). 
\end{array}
$$
Since $G$ is finitely presented, there exists a presentation: $\Hom _{\CC}(\  , N) \to \Hom _{\CC}(\ , M) \to G \to 0$. 
Applying $\Hom_{\AA / \SS} (\  ,  \overline{F})$ to the presentation, we obtain the following commutative diagram: 
$$
\xymatrix@C=0.6em@R=2em{
0\ar[r]&\Hom_{\AA/\SS}(\overline{G},\overline{F})\ar[r]\ar[d]&\Hom_{\AA/\SS}(\overline{\Hom_{\CC}(\ ,M)}, \overline{F} )\ar[r]\ar[d]_{\cong}&\Hom_{\AA/\SS}(\overline{\Hom_{\CC}(\ ,N)}, \overline{F}) \ar[d]^{\cong} \\
0\ar[r] &\Hom_{\ModCC}(G,\widetilde{F})\ar[r] &\Hom_{\ModCC}(\Hom_{\CC}(\ ,M),\widetilde{F})\ar[r]&\Hom_{\ModCC}(\Hom_{\CC}(\ ,N), \widetilde{F}).}
$$
The second and third vertical morphisms are isomorphisms, and hence the first vertical morphism is also an isomorphism in the assertion. 
\end{proof}

The following corollary follows immediately from Lemma \ref{A6}. 

\begin{corollary}\label{A7}
If $G \in \SS$, then $\Hom_{\ModCC} (G, \widetilde{F})=0$. 
\end{corollary}

For each $G \in \AA$, let $\eta_G : G \to \widetilde{G}$ be the natural morphism. 
We define $T: \widetilde{\AA } \to \AA/ \SS $ as follows. 
We set $T( \widetilde{F})=\overline{F}$. 
Note that $\overline{F}$ is determined uniquely. 
See the proof of Theorem \ref{A8} below. 
For a morphism $u: \widetilde{G} \to \widetilde{F}$, we consider $u \circ \eta_G : G \to \widetilde{F}$. 
By Lemma \ref{A6}, we have an isomorphism $\Phi_{G, F}: \Hom_{\AA /  \SS } ( \overline{G}, \overline{F}) \to \Hom_{\ModCC} (G, \widetilde{F})$. 
We define $T(u)$ to be the morphism $\Phi_{G, F}^{-1} (u \circ \eta_G)$.

\begin{theorem}\label{A8}
The functor $T$ is an equivalence of categories $\widetilde{\AA} \cong \AA/ \SS$. 
\end{theorem}

\begin{proof}
It is clear that $T$ is dense. 
We show that $T$ is fully faithful. 
Let $p: G \to G'$ be a pseudo-isomorphism. 
Since $\overline{p}$ is an isomorphism in $\AA / \SS$, the induced map $p^{\ast} :\Hom_{\ModCC }(G',\widetilde{F}) \to \Hom_{\ModCC} (G,\widetilde{F})$ is an isomorphism by Lemma \ref{A6}. 
Since this holds for every $p:G\to G'$ in $P_G$, the canonical morphism $\eta_G : G \to \widetilde{G}$ induces an isomorphism: 
\begin{align*}
\Hom_{\ModCC} (\widetilde{G}, \widetilde{F}) &= \Hom_{\ModCC} (\colim {}_{P_G}G', \widetilde{F}) \\
 &\cong \varprojlim {}_{{P_G}^{op}} \Hom_{\ModCC} (G', \widetilde{F}) \cong \Hom_{\ModCC}(G,\widetilde{F}).
\end{align*}
The isomorphism is induced by composition with $\eta_G$.
Therefore, using Lemma \ref{A6} again, 
$$
\Hom_{\ModCC} (\widetilde{G}, \widetilde{F}) \cong \Hom_{\ModCC} (G, \widetilde{F}) \cong \Hom_{\AA /  \SS } ( \overline{G}, \overline{F}).  
$$
This shows that $T$ is fully faithful.  
\end{proof}

For later use, we state the following lemma without proof. 

\begin{lemma}\label{A10}
Let $0\to L\to M \to  N \to 0$ be a short exact sequence in $\CC$. 
Suppose that $L\cong L'\oplus I$ and $N\cong N'\oplus P$, where $I$ is an injective object and $P$ is a projective object in $\CC$. 
Then the sequence is isomorphic to the direct sum of $0\to I \to  I\to0\to0$, $0 \to 0 \to P \to P \to 0$, and $0\to L'\to M'\to N' \to 0$ in $\CC$. 

In particular, for $F \in \AA$ with $0 \to \Hom_{\CC} (\  , L) \to \Hom_{\CC} (\  , M) \to \Hom_{\CC} (\   , N) \to F \to 0$, after replacing the representing short exact sequence, we may assume that $L$ has no injective direct summand and that $N$ has no projective direct summand.
\end{lemma}
\section{Isolated singularity}\label{B}

In this section, we investigate the relationship between the Krull--Gabriel dimension of $\modsCML$ and the singularities of the base algebra $\LL$. 

Let us recall the definition of Krull--Gabriel dimension for an abelian category $\AA$. 

\begin{definition}\cite[Definition 2.1]{G85}\label{B1}
Let $\AA$ be an abelian category. 
Define $\AA_{-1} = 0$. 
For each $n \geq 0$, let $\AA _{n}$ be the subcategory of all objects of finite length in $\AA/\AA_{n-1}$. 
We define $\KGdim \AA = \min \{ n \mid \AA = \AA _n \}$ if such a minimum exists, and set $\KGdim \AA = \infty $ otherwise.  
\end{definition}

In the rest of this paper, we always assume that $(R, \m_R)$ is a commutative complete Cohen--Macaulay local ring with residue field $k$ and a canonical module $\omega _R$. 
Since $R$ is complete, $R$ admits a canonical module.  
Let $\LL$ be an $R$-order, that is, $\LL$ is an $R$-algebra which is a finitely generated maximal Cohen--Macaulay (abbr. MCM) $R$-module. 
We denote by $\modR$ and $\modL$ the categories of finitely generated $R$-modules and finitely generated $\LL$-modules, respectively, with $R$-homomorphisms and $\LL$-homomorphisms as morphisms.
We also denote by $\CML$ the full subcategory of $\modL$ consisting of MCM $R$-modules. 
That is, 
$$
\CML =  \{ M \in \modL \mid \text{$M$ is an MCM $R$-module} \} .
$$
Since $R$ is complete, $\modR$, and hence $\modL$ and $\CML$, are  Krull--Schmidt categories (cf. \cite[(1.18)]{Y}). 
We denote by $\sCML$ the stable category of $\CML$.  
The objects of $\sCML$ are the same as those of $\CML$, and the morphism sets are $\sHomL(M, N):=\HomL(M, N)/ \{ M \to P \to N \text{ where $P$ is projective} \}$. 
We recall that the categories $\CML$ and $\sCML$ are coherent categories (\cite[$\S$ 4]{Y}).

From now on, we focus on the functor category $\modsCML$. 
We first recall that $\modsCML$ can be identified with a subcategory of $\modCML$.

\begin{remark}\label{B2}
We have an equivalence of categories $\modsCML \cong  \{ F \in \modCML \mid F(P)=0 \text{ for all projective $\LL$-modules $P$} \} $ by mapping $F \mapsto F \circ \iota $, where $\iota : \CML \to \sCML$ (see \cite[Remark 2.6]{Y05} or \cite[Lemma 2.5]{E19}). 
Under this equivalence, every object $F\in \modsCML$ is represented by a short exact sequence $0\to L\to M\to N\to 0$ in $\CML$, in the sense that $F$ admits a presentation $0\to \HomL (\ ,L)\to \HomL (\ ,M)\to \HomL (\ ,N)\to F\to 0$. 
Equivalently, $F$ is described as a functor on the stable category by a presentation $\sHomL (\ , L) \to \sHomL (\ , M) \to \sHomL (\ , N) \to F \to 0$ (cf. \cite[Remark 4.16]{Y}).  
\end{remark}

We say that $\LL$ is of {\it finite} CM representation type if there is only a finite number of isomorphism classes of indecomposable MCM $\LL$-modules.

\begin{remark}\label{B12}
We consider $\modsCML$ rather than $\modCML$ for the following reason. 
In the commutative case, namely when $\LL=R$ is a complete Cohen--Macaulay local ring and $\CML=\CMR$, the first author proves that $\KGdim \modsCML=0$ if and only if $\LL$ has a finite CM representation type \cite[Theorem 2.12]{H24}.
This is an analogue of the corresponding result for finite dimensional algebras \cite{A74}. 
On the other hand, even if $\LL$ is of finite CM representation type, $\KGdim \modCML$ need not be equal to $0$ \cite[Remark 2.13]{H24}.
Thus $\modsCML$ is a more natural category for investigating the relationship between Krull--Gabriel dimension and CM representation type. 
\end{remark}

\begin{definition}
Let $\LL$ be an $R$-order. 
\begin{enumerate}[\rm(1)]
\item We say that $\LL$ is an isolated singularity if $\gldim \LL_{\p} = \dim R_{\p}$ for every non-maximal prime ideal $\p \in \Spec R$.    

\item We say that $\LL$ is a Gorenstein order if $\omega_{\LL} := \HomR (\LL , \omega_R)$ is a projective $\LL$-module. 
\end{enumerate}
\end{definition}

\begin{remark}
Note that an $R$-order $\LL$ is an isolated singularity if and only if $\sHom_{\LL} (M, N)$ is of finite length as an $R$-module for any $M,N$ in $\CML$.  
\end{remark}

\begin{remark}
Let $\LL$ be a Gorenstein $R$-order. 
\begin{enumerate}[\rm(1)]
\item A finitely generated $\LL$-module $M$ is MCM if and only if $\Ext^i _\LL ( M , \omega_{\LL})=0$ for all $i >0$. 
In particular, this induces the isomorphism $\sHom_{\LL} (M, N)\cong \Ext^1 _{\LL}(M, \Omega N)$ for $M$, $N$ in $\CML$. 

\item The $R$-order $\LL$ is a Gorenstein $R$-order if and only if $\CML$ is a Frobenius category. 
Equivalently, the projective and injective objects in $\CML$ coincide. 
\end{enumerate}
\end{remark}

Let us recall the theory of Auslander--Reiten (abbr. AR) sequences in $\CML$. 
For the details, we recommend that the reader refer to \cite{AR87, I08, LW12, Y}. 

\begin{theorem}\cite{AR87, Nak22}
Let $N \in \CML$ be a nonprojective indecomposable object.
Then the following are equivalent. 
\begin{enumerate}[\rm(1)]
\item There exists an AR sequence $0 \to L \to M \to N \to 0$ in $\CML$. 
\item $N_{\p}$ is $\LL_\p$-projective for all $\p$ in the punctured spectrum $\Spec R \setminus \{ \m_R \}$. 
\end{enumerate}
Let $L \in \CML$ be a noninjective indecomposable object.
Then the following are equivalent. 
\begin{enumerate}[\rm(1')]
\item There exists an AR sequence $0 \to L \to M \to N \to 0$ in $\CML$. 
\item $L_{\p}$ is injective in $\CC (\LL_{\p})$ for all $\p \in \Spec R \setminus \{\m_R \}$. 
\end{enumerate} 
\end{theorem}

We denote by $\PP_0$ the full subcategory of $\CML$ consisting of MCM $\LL$-modules $N$ such that $N_{\p}$ is $\LL_\p$-projective for all $\p \in \Spec R \setminus \{ \m_R \}$. 
We also denote by $\II_0$ the full subcategory of $\CML$ consisting of MCM $\LL$-modules $L$ such that $L_{\p}$ is injective in $\CC (\LL_{\p})$ for all $\p \in \Spec R \setminus \{\m_R \}$. 

\begin{remark}\label{B3}
Let $S$ be a simple functor in $\modCML$. 
Then there exists an AR sequence $0 \to L \to M \to N \to 0$ such that $0 \to \HomL (\   , L) \to \HomL (\   , M) \to \HomL (\   , N) \to S \to 0$. 
All simple objects of $\modCML$, hence in $\modsCML$,  are obtained in this way from AR sequences. 
See \cite[(4.12)]{Y} or \cite[Proposition 2.3]{E19}. 
\end{remark}

Let $F \in \modsCML$ admit a presentation $0 \to \HomL (\  , X) \to \HomL (\  , Y) \to \HomL (\   , Z) \to F \to 0$ and $S_N$ be a simple functor admitting a presentation $0 \to \HomL (\  , L) \to \HomL (\  , M) \to \HomL (\   , N) \to S_N \to 0$. 
Suppose that there is a monomorphism $S_N \to F$ in $\modsCML$. 
Then we have a commutative diagram: 
$$
\xymatrix{
0 \ar[r] & \HomL(\ ,X) \ar[r] & \HomL(\ ,Y) \ar[r] & \HomL(\ ,Z) \ar[r] & F \ar[r] & 0 \\
0 \ar[r] & \HomL(\ ,L) \ar[r] \ar[u] & \HomL(\ ,M) \ar[r] \ar[u] & \HomL(\ ,N) \ar[r] \ar[u] & S_N \ar[r] \ar[u] & 0
}
$$
By Yoneda's lemma, we obtain a commutative diagram:
$$
\begin{CD}
0 @>>>X @>>> Y @>>> Z @>>> 0 \\
@. @AA{h}A @AA{g}A @AA{f}A  @. \\
0 @>>> L @>>> M @>>> N @>>> 0.  
\end{CD}
$$

Similarly, suppose that there is an epimorphism $F \to S_N$ in $\modsCML$. 
By the dual argument above, we also obtain a commutative diagram:
$$
\begin{CD}
0 @>>>X @>>> Y @>>> Z @>>> 0 \\
@. @VV{u}V @VV{v}V @VV{w}V  @. \\
0 @>>> L @>>> M @>>> N @>>> 0.  
\end{CD}
$$

\begin{lemma}\label{B4}
\begin{enumerate}[\rm(1)]
\item A morphism $S_N \to F$ is nonzero if and only if $h$ is a split monomorphism. 

\item A morphism $F \to S_N$ is nonzero if and only if $w$ is a split epimorphism. 
\end{enumerate}
\end{lemma}

\begin{proof}
(1) Suppose that the morphism $S_N \to F$ is zero. 
Then there exists a morphism $k: N \to Y$ that makes the triangle in the top right corner commutative: 
$$
\xymatrix{
0 \ar[r] & X \ar[r] & Y \ar[r] & Z \ar[r] & 0 \\
0 \ar[r] & L \ar[r]^i \ar[u]^h & M \ar[r] \ar[u]^g 
& N \ar[r] \ar[u]^f \ar[ul]^k & 0 .
}
$$
In particular, there exists a morphism $\ell : M \to X$ such that $h = \ell \circ i$.
Assume, for a contradiction, that $h$ is a split monomorphism. 
Then there exists a morphism $p : X \to L$ such that $p \circ h = 1_L$. 
Hence
$$
1_L = p \circ h = p \circ (\ell \circ i) = (p \circ \ell) \circ i.
$$
It follows that $i$ is a split monomorphism, which is a contradiction. 
Therefore $h$ is not a split monomorphism.

Conversely, suppose that $h$ is not a split monomorphism. 
Since $i$ is left almost split, there exists a morphism $\ell : M \to X$ such that $h = \ell \circ i$.
Applying $\Ext_\LL^1(\ , \ )$, we obtain the following commutative diagram:
$$
\xymatrix{
0 \ar[r] & F \ar[r] & \Ext_\LL^1(\ , X) \ar[r] & \Ext_\LL^1(\ , Y) \\
0 \ar[r] & S_N \ar[r] \ar[u] & \Ext_{\LL}^1(\ , L) \ar[r] \ar[u]^{h^\ast} & \Ext_\LL^1(\ , M) \ar[u]^{g^\ast} \ar[ul]^{\ell^\ast}.
}
$$
Thus the morphism $S_N \to F$ is zero. 

(2) Suppose that $w$ is not a split epimorphism. 
Then there exists a morphism $s: Z \to M$ that makes the triangle in the top right corner commutative: 
$$
\xymatrix{
0 \ar[r] & X \ar[r]\ar[d]^u  & Y \ar[r]\ar[d]^v& Z \ar[r]\ar[d]^w \ar[dl]_s & 0 \\
0 \ar[r] & L \ar[r] & M \ar[r]^j & N \ar[r]  & 0 .
}
$$
Applying $\HomL(\ , \ )$ to the above diagram, we have a commutative diagram:
$$
\xymatrix{
& \HomL(\ , Z) \ar[r]\ar[d]^{\HomL(\ , w)} \ar[dl]_{\HomL(\ , s)} & F\ar[r]\ar[d]^{\Psi}& 0 \\
 \HomL(\ , M) \ar[r]_{\smash{\lower2.3ex\hbox{$\scriptstyle\HomL(\ ,j)$}}}& \HomL(\ , N) \ar[r]  & S_N \ar[r]& 0.
}
$$
This diagram shows that $\Psi = 0$. 

Conversely, suppose that $w$ is a split epimorphism and assume, for a contradiction, that $\Psi : F \to S_N$ is zero. 
Since $\Psi(Z) =0$, there exists an $\End_{\LL} (Z)^{op}$-module homomorphism $\delta : \HomL (Z, Z) \to \HomL(Z, M)$ that makes the left triangle commutative: 
$$
\xymatrix{
& \HomL(Z , Z) \ar[r]\ar[d]^{\HomL(Z , w)}\ar[dl]_{\delta} & F(Z)\ar[r]\ar[d]^{\Psi(Z)=0}& 0 \\
 \HomL(Z, M) \ar[r]_{\smash{\lower2.0ex\hbox{$\scriptstyle\HomL(Z, j)$}}}& \HomL(Z , N) \ar[r]  & S_N(Z) \ar[r]& 0.
}
$$
Set $\ell = \delta (1_Z)$, and then we have $w= j \circ \ell$. 
Since $w$ is a split epimorphism, there exists $r: N \to Z$ such that $w \circ r = 1_N$. 
Hence,
$$
1_N = (j \circ \ell) \circ r = j \circ (\ell \circ r). 
$$
This is a contradiction since $j$ is not a split epimorphism. 
Therefore, $\Psi$ is nonzero. 
\end{proof}

In particular, $F$ has no simple subfunctors (resp. quotient functors) if and only if $X$ (resp. $Z$) has no indecomposable direct summand belonging to $\II_0$ (resp. $\PP_0$). 

For a short exact sequence $0\to X \xrightarrow{f} Y \xrightarrow{g} Z \to 0$ in $\CML$, we say that this sequence is minimal if it has no nonzero split exact direct summand. 
Equivalently, it has no direct summand of the form $0 \to 0 \to W \xrightarrow{1_W} W \to 0$ or $0 \to W' \xrightarrow{1_{W'}} W' \to 0 \to 0$.  
Since $\modR$ is a Krull--Schmidt category, every short exact sequence of finitely generated modules can be made minimal. 
It is well known that, in a Krull--Schmidt category, a short exact sequence is minimal if and only if $f$ is left minimal and $g$ is right minimal in the sense of \cite[$\S$1]{ARS}.
Moreover, we may also assume that $X$ (resp. $Z$) does not have injective (resp. projective) direct summands by Lemma \ref{A10}. 
Let $F \in \modsCML$ admit a presentation $0 \to \HomL (\  , X) \to \HomL (\  , Y) \to \HomL (\   , Z) \to F \to 0$. 
We say that $F \in \modsCML$ is minimal if the short exact sequence $0 \to X \to Y \to Z \to 0$ is minimal in the above sense.

\begin{proposition}\label{B5}
Let $F \in \modsCML$ be minimal with a presentation $0 \to \HomL (\  , X) \to \HomL (\  , Y) \to \HomL (\   , Z) \to F \to 0$. 
The number of isomorphism classes of simple subfunctors of $F$ is equal to the number of isomorphism classes of indecomposable direct summands of $X$ belonging to $\II_0$. 
Dually, the number of isomorphism classes of simple quotient functors of $F$ is equal to the number of isomorphism classes of indecomposable direct summands of $Z$ belonging to $\PP_0$. 
\end{proposition}

\begin{proof}
We prove the initial statement. 
Let $S_N$ be a simple functor which corresponds to $0 \to L \to M \to N \to 0$. 
By Lemma \ref{B4}, the morphism $S_N \to F$ is nonzero if and only if $L$ is a direct summand of $X$. 
Since such an object $L$ belongs to $\II_0$, the assertion follows.
\end{proof}

For simplicity, we denote $\AA = \modsCML$ and let $\AA_0$ be the full subcategory consisting of finite length functors. 
Note that $\AA_0$ is a Serre subcategory. 
It is worth noting that since $\AA$ and $\AA / \AA_0$ are $R$-linear categories, their Hom sets naturally have $R$-module structure.

\begin{lemma}\label{B7}
Suppose that $\overline{S} \in \AA/ \AA_0$ is simple and that $S$ has a presentation $\sHomL(\ ,  X) \to S \to 0$ in $\AA$. 
Then we may assume that $X$ is indecomposable in $\sCML$. 
\end{lemma}

\begin{proof}
Suppose that $\sHomL (\ ,  X_1 \oplus X_2) \to S \to 0$. 
Since $\sHomL (\ ,  X_1 \oplus X_2) \cong \sHomL (\ ,  X_1) \oplus \sHomL (\ ,  X_2)$, one of the images maps onto $\overline{S}$. 
\end{proof}

\begin{proposition}\label{B8}
Suppose that $\overline{S} \in \AA/ \AA_0$ is simple and that $S$ has a presentation $\sHomL(\ ,  X) \to S \to 0$ in $\AA$ where $X$ is indecomposable. 
\begin{enumerate}[\rm(1)]
\item If $X \in \PP_0$, then $\End_{\AA/\AA_0} (\overline{S})$ is an $R$-algebra of finite length. 
In particular, $\m_R ^n \overline{S} = 0$ for $n \gg 0$.  
\item If $X \not\in \PP_0$, then $\Hom_{\AA/ \AA_0} (\overline{S}, \overline{G})$ is a finitely generated $R$-module for all $G \in \AA$. 
\end{enumerate}
\end{proposition}

\begin{proof}
(1) Suppose that $X$ belongs to $\PP_0$. 
Then $\sHomL (X, X)$ is an $R$-module of finite length, so that $\m_R ^n \sHomL (X, X) = 0$ for some $n\gg 0$. 
Hence, $\m_R ^n \sHomL(\ , X)=0$. 
This implies that $\m_R ^n$ annihilates $S$, and hence $ \m_R ^n \overline{S} = 0$.

(2) Suppose that $X$ does not belong to $\PP_0$. 
By Proposition \ref{B5}, $T(X) = 0$ for all simple functors $T \in \AA$. 
Thus $L(X) = 0$ for every functor $L \in \AA_0$ by Proposition \ref{B5}. 
In particular, $G(X) \cong G'(X)$ if $G$, $G'$ are pseudo-isomorphic for $G, G' \in \AA$. 
Combining this with Yoneda's lemma,   
$$
\Hom_{\AA / \AA_0} (\overline{\sHomL (\ ,  X)}, \overline{G}) = \colim_{P_G \ni (G \to G')} \Hom_{\AA} ({\sHomL (\ ,  X)}, G') \cong G(X). 
$$
Since we have $\overline{\sHomL (\ ,  X)} \to \overline{S} \to 0$ in $\AA / \AA_0$, we see that $\Hom_{\AA / \AA_0} (\overline{S}, \overline{G})$ is an $R$-submodule of $\Hom_{\AA / \AA_0} (\overline{\sHomL (\ ,  X)}, \overline{G}) \cong G(X)$. 
Since $G(X)$ is a finitely generated $R$-module, $\Hom_{\AA / \AA_0} (\overline{S}, \overline{G})$ is also finitely generated as an $R$-module. 
\end{proof}

The following is a crucial lemma in this section. 

\begin{lemma}\label{B9}
Let $\overline{S} \in \AA / \AA_0$ be a simple functor. 
Then $\End _{\AA / \AA_0} (\overline{S})$ is a division $k$-algebra, where $k=R/\m_R$. 
In particular, $\m_R \overline{S} = 0$.  
\end{lemma}

\begin{proof}
By Lemma \ref{B7}, we may take a presentation $\sHomL (\ ,X) \to S\to 0$ with $X$ indecomposable.
Recall that $\End_{\AA / \AA_0} (\overline{S})$ is a division ring since $\overline{S}$ is simple in $\AA / \AA_0$. 
Let $\p = \{ a \in R \mid a\overline{S}=0\}$. 
Here, $a\overline{S}=0$ means the natural transformation $\overline{S} \to \overline{S}$ induced by $a$ is zero.  
Then $\p$ is a prime ideal. 
Moreover, for any $a \in R \setminus \p$, multiplication by $a$ induces an automorphism of $\overline{S}$. 
Consequently, we have a canonical embedding $\kappa (\p )= R_{\p}/\p R_{\p} \to \End_{\AA / \AA_0} (\overline{S})$.

Suppose that $X \in \PP_0$. 
Since $\m_R ^n \overline{S}=0$, we have $\m_R ^n \kappa (\p )=0$. 
Therefore, $\p = \m_R$.
Suppose that $X \not\in \PP_0$. 
Then $\End_{\AA / \AA_0} (\overline{S})$ is a finitely generated $R$-module. 
Hence, $\kappa (\p )$ is also a finitely generated $R$-module. 
Note that $\kappa (\p )$ is finitely generated as an $R/ \p$-module. 
Since $R/ \p$ is an integral domain, one shows that $R/ \p$ is a field. 
See \cite[Proposition 5.7]{AM}. 
Consequently, $\p =\m_R$.  
Therefore, $\m_R \overline{S} = \p \overline{S} = 0$. 
\end{proof}

We now state the main theorem in this section.

\begin{theorem}\label{B10}
Suppose that $\LL$ is a Gorenstein $R$-order. 
If $\KGdim \AA \le 1$, then $\LL$ is an isolated singularity. 
\end{theorem}

\begin{proof}
Assume that $\LL$ is not an isolated singularity. 
Then there exists an indecomposable $M \in \CML$ which is not projective on the punctured spectrum.  
We consider the functor $F$ which corresponds to the exact sequence $0 \to \Omega M \to P \to M \to 0$, where $P$ is a projective $\LL$-module. 
That is, $F \cong \sHomL (\ ,  M)$. 
We may assume that $M$(resp. $\Omega M$) has no projective (resp. injective) summands. 
Since $\LL$ is Gorenstein, we have $F  \cong  \Ext ^1 _\LL (\ ,  \Omega M)$. 
Moreover $\Omega M$ also does not belong to $\II_0$. 
Note that $\PP_0 = \II_0$ since $\LL$ is Gorenstein. 
Since $\Omega$ gives an autofunctor on $\sCML$, $\Omega M$ is indecomposable. 
Thus, if $\Omega M$ belongs to $\PP_0 = \II_0$, then so does $M$. 
This contradicts that $M \not\in \PP_0$. 
Hence, by Proposition \ref{B5}, $F$ belongs to $\AA_0 ^{\perp} \cap {}^{\perp}\! A_0$. 
Thus $\End _{\AA / \AA_0} (\overline{F}) \cong \End_\AA (F) \cong \sHomL (M, M)$. 
The last isomorphism follows from Yoneda's lemma. 
This shows that  $\End _{\AA / \AA_0} (\overline{F})$ is finitely generated as an $R$-module. 

Since $\KGdim \AA \le 1$, $\overline{F}$ is of finite length in $\AA / \AA_0$. 
There is a filtration
$$
0 = \overline{F}_0 \subset \overline{F}_1 \subset \overline{F}_2 \subset \cdots \subset \overline{F}_{n-1} \subset \overline{F}_n= \overline{F} 
$$
such that $\overline{S}_i := \overline{F}_i/\overline{F}_{i-1}$ is simple for each $i$. 
By Lemma \ref{B9}, we have $\m_R ^n \overline{F}=0$, and hence $\m_R ^n \sHomL (M, M) =0$. 
This implies that the length of $\sHomL (M, M)\cong \Ext ^1 _\LL (M, \Omega M)$ is finite as an $R$-module.
Hence, $$\Ext ^1 _\LL (M, \Omega M)_\p \cong \Ext ^1 _{\LL_\p} (M_\p, (\Omega M )_\p) =0$$ for all $\p \in \Spec R \setminus \{\m_R \}$. 
This shows that $M$ is projective on the punctured spectrum, which is a contradiction.
\end{proof}

\begin{remark}\label{B11}
As mentioned in Remark \ref{B12}, the first author shows that $\KGdim \modsCML=0$ if and only if $\LL$ is of finite CM representation type. 
By \cite[Proposition 3.3]{E19}, this equivalence also holds for $R$-orders. 
Hence every such ring is an isolated singularity (cf. \cite[2.2]{Y}). 
The first author also gives examples for which $\KGdim \modsCML=2$ \cite{H24}. 
These examples are hypersurface rings that are not isolated singularities. 
Thus Theorem \ref{B10} shows that these examples attain the smallest possible value of the Krull--Gabriel dimension among non-isolated singularities.
\end{remark}

\section{Krull--Gabriel dimension}\label{C}

This section is devoted to the proof of the following theorem, which is analogous to the corresponding result for finite dimensional algebras. 
We say that $\LL$ is of  {\it countable} CM representation type if there exist infinitely many but only countably many isomorphism classes of indecomposable MCM $\LL$-modules. 
We also say that $\LL$ is of  {\it uncountable} CM representation type if there exist uncountably many isomorphism classes of indecomposable MCM $\LL$-modules.

\begin{theorem}\label{C1}
Let $\LL$ be an isolated singularity. 
Suppose that $\LL$ is of uncountable CM representation type. 
Then $\KGdim \modsCML > 1$. 
\end{theorem}

In what follows, we always assume that $\LL$ is an {\it isolated singularity}.   
Then every indecomposable MCM $\LL$-module admits an AR sequence. 

To show Theorem \ref{C1}, we need several preparations. 
First, we will introduce the socle series and the radical series for functors in a manner similar to that for modules. 

The main difficulty is to control the support of a simple object in $\AA/\AA_0$. 
Instead of passing to generic or pure-injective objects, we develop our argument from the structure of finitely presented functors. 
To this end, we introduce two constructions associated with a finitely presented functor, namely the direct limit of its socle series and the inverse limit of its radical series. 
These constructions enable us to control the support of simple objects in the quotient category. 

Let $G \in \ModsCML$. 
We denote by $\soc G$ the largest semisimple subfunctor, i.e. the sum of all the simple subfunctors, of $G$. 
For a nonnegative integer $i$, we define $\soc ^i G$ as $\soc ^0 G = 0$ and
$
\soc ^{i+1} G = p ^{-1}(\soc (G/ \soc ^i G)),
$
where $p : G \to G/\soc ^i G$ is the natural projection. 
Then we define the socle series of $G$ by the following increasing sequence of subfunctors: 
$$
0 = \soc ^0 G \subset \soc ^1 G = \soc G \subset \soc ^2 G \subset \cdots \subset G. 
$$

Let $F \in \ModsCML$. 
We denote by $\top F$ the maximal semisimple quotient of $F$.  
Thus there is a canonical epimorphism $\pi : F \to \top F$ such that every morphism $f : F \to S$, where $S$ is semisimple, factors uniquely through $\pi$. 
We define $\rad F = \Ker \pi$. 
Note that if $F$ belongs to $\modsCML$, then $\rad F$ also belongs to $\modsCML$. 
We define $\rad^{0} F = F$ and $\rad^{i} F = \rad (\rad^{i-1} F)$ for $i \geq 1$, and the radical series of $F$ by the following decreasing sequence of subfunctors: 
$$
F = \rad^0 F \supset \rad F \supset \rad ^2 F \supset \cdots .
$$

We note that the socle and the top exist for every $F \in \modsCML$. 
Let $F$ be represented by a minimal short exact sequence $0 \to X \to Y \to Z \to 0$. 
By Proposition \ref{B5}, the simple subfunctors of $F$ correspond to the indecomposable direct summands of $X$ belonging to
$\II_0$, whereas the simple quotient functors of $F$ correspond to the indecomposable direct summands of $Z$ belonging to $\PP_0$. 
Since $\CML$ is a Krull--Schmidt category, there are only finitely many such direct summands. 
Therefore, the sum of all simple subfunctors of $F$ is a semisimple functor of finite length, and the maximal semisimple quotient of $F$ is also a finite direct sum of simple functors. 
In particular, both belong to $\modsCML$.
Thus, $\soc F$ is the direct sum of the simple subfunctors corresponding to the indecomposable direct summands of $X$ that belong to $\II_0$, whereas $\top F$ is the direct sum of the simple quotient functors corresponding to the indecomposable direct summands of $Z$ that belong to $\PP_0$. 
If no such direct summand occurs, we set the corresponding semisimple functor equal to zero. 
Since $\modsCML$ is abelian and $\top F$ belongs to $\modsCML$, we have $\rad F \in \modsCML$. 

We denote $\AA = \modsCML$, and by $\AA_0$, we denote the full subcategory consisting of finite length functors, as in Section \ref{B}.  

\begin{lemma}\label{C4}
Let $F \in \AA$. The following statements hold.
\begin{enumerate}[\rm(1)]
\item If $F \neq 0$, then both $\soc F$ and $\top F$ are nonzero. 

\item If $F \in \AA \setminus \AA_0$, then $\rad F$ is nonzero. 

\item If $F\in \AA\setminus \AA_0$, then $\soc ^n F \subsetneq \soc ^{n+1} F$ and $\rad ^{n+1} F \subsetneq  \rad^{n} F$ for all $n\geq 0$.

\item Let $G$ be a subfunctor of $\rad F$. Then $\rad (F/G) = (\rad F)/G$. 
\end{enumerate}
\end{lemma}

\begin{proof}
(1) It follows from Proposition \ref{B5}. 

(2) If $\rad F=0$, then $F = \top F$. 
Thus $F$ is semisimple. 
Since $F\in \AA$, $F$ is a finite direct sum of simple functors. 
Therefore $F\in \AA_0$. 
This contradicts $F\in \AA\setminus \AA_0$. 
Thus $\rad F\neq 0$.

(3) Suppose that the socle series or the radical series of $F \in \AA$ terminates. 
If the socle series terminates, then $F' = \cup_{n \ge 0} \soc^n F \in \AA_0$ and $F/F'$ has no socle. 
Since $F/F' \in \AA$, this contradicts (1). 
If the radical series terminates, then $F'' = \cap_{n \ge 0} \rad^n F \in \AA$ and $\top F'' =0$. 
This also contradicts (1). 

(4) Let $q: F\to F/G$ be the canonical epimorphism. 
Since $G\subseteq \rad F$, the canonical epimorphism $F\to \top F$ factors through $q$. 
Hence we have an epimorphism $F/G \to \top F$ whose kernel is $(\rad F)/G$.
Since $\top F$ is semisimple, this epimorphism factors through the maximal semisimple quotient $F/G\to \top (F/G)$. 
Therefore $\rad (F/G) = \Ker(F/G \to \top(F/G)) \subseteq \Ker(F/G \to \top F) =(\rad F)/G$. 

To show the reverse inclusion, let $p : F/G \to \top (F/G)$ be the canonical epimorphism. 
Recall that $\Ker p = \rad (F/G)$. 
Since $\top (F/G)$ is semisimple, $\rad F \subseteq \Ker (p \circ q)$. 
We obtain $\Ker (p \circ q) = q^{-1} (\Ker p ) = q^{-1} (\rad (F/G))$. 
Therefore $q(\rad F) = (\rad F)/ G \subseteq \rad (F/G)$. 
\end{proof}

We define the assignments associated with finitely presented functors, which will play key roles in the sequel. 

\begin{definition}\label{C5}
Let $F \in \ModsCML$. 
We define $\Gamma (F)$ as $\varinjlim _n \soc ^{n} F$ for the socle series of $F$. 
For the radical series of $F$, we define $\Theta (F)$ as $\varprojlim _n \rad ^{n} F$. 

Equivalently, we denote $\Gamma (F) = \bigcup _n \soc ^n F$ and $\Theta (F) = \bigcap_n \rad^n F$.
\end{definition}

We record the following facts. 

\begin{proposition}\label{C6}
Let $F \in \AA$. The following statements hold.
\begin{enumerate}[\rm(1)]
\item If $F \in \AA \setminus \AA_0$, then $\Theta (F) \cap \soc F \neq 0$, in particular $\Theta (F) \neq 0$. 
Moreover, $\Theta (F) \notin \AA$.  

\item $\Gamma(\Theta(F)) \subseteq \Theta(F)$ and $\Gamma(\Theta(F)) \subseteq \Gamma(F)$. 

\item If $F \in \AA \setminus \AA_0$, then  $\Gamma(\Theta(F)) \neq 0$. 
\end{enumerate}
\end{proposition}

\begin{proof}
(1) Let $F \in \AA\setminus \AA_0$. 
Since $F \in \AA$, $\soc F$ is a finite direct sum of simple functors. 
Since $\rad^i F \in \AA$ for all $i$, it has a nonzero socle. 
Moreover, every simple subfunctor of $\rad^i F$ is a simple subfunctor of $F$. 
That is, 
$$
0 \neq \soc (\rad ^i F) \subseteq \rad ^i F \cap \soc F 
$$
for all $i$. 
This gives a decreasing chain of subobjects of $\soc F$; $\{ \rad ^i F \cap \soc F \} _{i \in \mathbb{N}}$. 
Since $\rad ^i F \cap \soc F$ is nonzero for all $i$ and is of finite length, $\rad ^i F \cap \soc F = \rad ^{i+1} F \cap \soc F$ for $i \gg 0$. 
Therefore $\Theta (F) \cap \soc F \neq 0$.

Assume that $\Theta(F)$ belongs to $\AA$. 
Then $F/\Theta(F)$ belongs to $\AA$. 
By the definition of $\Theta(F)$, $\Theta(F) \subseteq \rad^i F$ for all $i$. 
By Lemma \ref{C4}(4), $\rad ^i (F/ \Theta (F)) = (\rad^i F)/ \Theta (F)$. 
It follows that $\Theta(F/\Theta(F))=0$. 
By Lemma \ref{C4}(3) and (4), $\rad^{i+1}(F)/\Theta(F) \subsetneq \rad^{i}(F)/\Theta(F)$. 
Since $F \in \AA\setminus \AA_0$, Lemma \ref{C4}(3) also shows that the radical series of $F/\Theta(F)$ does not terminate. 
This shows that $F/\Theta(F)$ belongs to $\AA \setminus \AA_0$. 
Together with the first part of the proof, this implies $\Theta(F/\Theta(F))\neq 0$, which is a contradiction. 

(2) This follows from the definition. 

(3) Let $F\in \AA\setminus \AA_0$. 
By (1), we have $\Theta(F)\cap \soc F \neq 0$. 
Therefore $\Theta(F)\cap \soc F \subseteq \soc \Theta(F) \subseteq \Gamma(\Theta(F))\neq 0$. 
This shows (3).     
\end{proof}

\begin{theorem}\label{C7}
Let $F \in \AA \setminus \AA_0$. 
Assume that $\overline{F}$ is simple in $\AA /\AA_0$. 
Then $\Gamma (\Theta (F)) = \Theta (F)$.  
\end{theorem}

\begin{proof}
Assume, for a contradiction, that $\Gamma(\Theta(F)) \subsetneq \Theta(F)$. 
Then one has  $\Theta(F)/\Gamma(\Theta(F))\neq 0$. 
Hence there exists an indecomposable MCM $\LL$-module $X\in \sCML$ such that $\left(\Theta(F)/\Gamma(\Theta(F))\right)(X)\neq 0$. 
By Yoneda's lemma, there exists a nonzero morphism $f\in \Hom_{\ModsCML} \left( \sHomL(\ , X), \Theta(F)/\Gamma(\Theta(F)) \right)$. 
Since $\sHomL(\ , X)$ is projective, $f$ lifts to a morphism $\widetilde f:\sHomL(\ , X)\to \Theta(F)$. 
Thus we have the following commutative diagram:
$$
\xymatrix{
\sHomL(\ , X) \ar[r]^-{f\neq 0} \ar[dr]_-{\widetilde f} &\Theta(F)/\Gamma(\Theta(F)) \\
&\Theta(F) \subset F . \ar@{->>}[u]}
$$
Put $G=\Im \widetilde f$. 
Then $G\in \AA$. 
Indeed, $G$ is the image of the composite morphism $\sHomL(\ , X) \xrightarrow{\widetilde f} \Theta(F) \hookrightarrow F$. 
Since $\AA$ is an abelian category, $G$ belongs to $\AA$.  
Note that $G$ is a subfunctor of $F$. 
By Corollary \ref{A4}, we have either $G\in \AA_0$ or $F/G\in \AA_0$. 
If $G \in \AA_0$, then $G$ has finite length. 
Since $G \subseteq \Theta (F)$, we have $G\subseteq \soc^n \Theta (F)$ for some $n$, and hence $G \subseteq \Gamma (\Theta (F))$.
Hence the composite $\sHomL(\ , X) \xrightarrow{\widetilde f} \Theta(F) \twoheadrightarrow \Theta(F)/\Gamma(\Theta(F))$ is zero. 
This composite is precisely $f$, which contradicts $f\neq 0$.
Thus we must have $F/G\in \AA_0$. 
Then $\rad ^n F \subseteq G$ for $n \gg 0$. 
Since $G \subseteq \Theta (F)$, $\rad ^n F = \Theta (F)$. 
This contradicts Lemma \ref{C4}(3). 
\end{proof}

For a functor $F \in \AA$, we denote by $\Supp F$ the set of isomorphism classes of indecomposable MCM $\LL$-modules $M$ such that $F(M) \not =0$: 
$$
\Supp F = \{ \text{an indecomposable MCM $\LL$-module $M$} \mid F(M) \not=0 \} /\cong . 
$$ 

\begin{corollary}\label{C8}
Assume that $\overline{F}$ is nonzero and simple in $\AA /  \AA_0$. 
Then $\Supp F$ is countable.  
\end{corollary}

\begin{proof}
Let $F \in \AA$ be a functor such that $\overline{F}$ is nonzero and simple in $\AA /  \AA_0$. 
Consider the radical series of $F$ and the socle series of $\Theta (F)$:
\begin{align*}
0 \subset \soc ^1 \Theta (F) \subset \soc ^2 \Theta (F) &\subset \cdots  \subset \Gamma (\Theta (F)) \\
=\Theta (F) &\subset \cdots \subset \rad ^2 F \subset \rad ^1 F \subset F. 
\end{align*}
Recall that, by Theorem \ref{C7}, $\Gamma (\Theta (F))=\Theta (F)$.  
Each successive quotient is semisimple, hence is a finite direct sum of simple functors and has finite support. 
Since there are only countably many successive quotients, $\Supp F$ is a countable set. 
\end{proof}

\begin{lemma}\label{C9}
Suppose that $F$ is pseudo-isomorphic to $G$, that is, $\overline{F} \cong \overline{G}$ in $\AA / \AA_0$. 
Then the symmetric difference of $\Supp F$ and $\Supp G$ is a finite set. 
In particular, if $\Supp F$ or $\Supp G$ is an infinite set, then the cardinalities of $\Supp F$ and $\Supp G$ are the same.   
\end{lemma}

\begin{proof}
Suppose that there is a pseudo-isomorphism $f : F \to G$. 
We have an exact sequence $0 \to K \to F \to G \to C \to 0$ in $\AA$ such that $K$ and $C$ are in $\AA_0$. 
Here $K = \Ker f$ and $C=\Coker f$. 
We also have the exact sequences: 
$$
0 \to K \to F \to \Im f \to 0, \quad 0 \to \Im f \to G \to C \to 0. 
$$
It is easy to see that the symmetric difference of $\Supp F$ and $\Supp G$ is contained in $\Supp K \cup \Supp C$. 
Notice that $\Supp K$ and $\Supp C$ are finite sets. 
Hence the symmetric difference is also a finite set. 

If $\Supp F$ or $\Supp G$ is infinite, then they differ by only finitely many indecomposable MCM modules. 
Therefore, the cardinalities of $\Supp F$ and $\Supp G$ are the same.
\end{proof}

\begin{lemma}\cite[Lemma 2.1]{H17}\cite[Proposition 4.7]{E19}\label{C12}
Let $\LL$ be an $R$-order. 
Then there exists $X \in \CML$ such that $\sHomL (M, X) \not= 0$ for each $M \in \CML$ with $\underline{M} \not= \underline{0}$ in $\sCML$. 
\end{lemma}

\begin{proof}
Recall that $\CML$ is a contravariantly finite subcategory of $\mod \LL$. 
Take a right $\CML$-approximation of $\LL / \rad \LL$ as $X$.  
\end{proof}

\begin{proof}[Proof of Theorem \ref{C1}] 
By Remark \ref{B11}, $\KGdim \modsCML = 0$ if and only if $\LL$ is of finite CM representation type. 
Hence, it is sufficient to consider the case $\KGdim \modsCML = 1$. 
We take $X$ in Lemma \ref{C12} and consider $ \sHomL(\ ,  X)$. 
Then $\overline{ \sHomL (\ , X) }$ has finite length in $\AA /\AA_0$. 
We have a composition series in $\AA /\AA_0$:
$$
0 = \overline{H_0} \subset \overline{H_1} \subset \overline{H_2} \subset \cdots  \subset \overline{H_n} = \overline{ \sHomL (\ , X) }. 
$$
For each $i$, we have a short exact sequence $0 \to \overline{H_{i-1}} \to \overline{H_{i}} \to \overline{S_i} \to 0$ where $\overline{S_i}$ is simple. 
There exists a short exact sequence $0 \to H'_{i-1} \to H_{i} \to S'_i \to 0$ in $\AA$ such that $\overline{H_{i-1}} \cong \overline{H'_{i-1}}$ and $\overline{S_{i}} \cong \overline{S'_{i}}$ in $\AA / \AA_0$ by Proposition \ref{A2}(2). 
By Lemma \ref{C9}, $\Supp H_i$ and $\Supp H'_i$ differ only by a finite set. 
Hence, by induction on $i$, the support of $ \sHomL(\ , X)$ is obtained from the supports of the $S'_i$ by adding or removing only finitely many indecomposable MCM modules.
By Corollary \ref{C8}, each $\Supp S' _i$ is countable. 
Therefore $\Supp  \sHomL (\ , X)$ is countable. 
By the choice of $X$, we conclude that the set of isomorphism classes of nonprojective indecomposable MCM $\LL$-modules is countable, so that $\LL$ is of countable CM representation type. 
This contradicts the assumption that $\LL$ is of uncountable CM representation type. 
Hence $\KGdim \modsCML \neq 1$.
\end{proof}

By Theorem \ref{B10}, if $\LL$ is Gorenstein, the assumption $\KGdim \modsCML = 1$ implies that $\LL$ is an isolated singularity. 
Therefore we also have the following.

\begin{corollary}\label{C10}
Let $\LL$ be a Gorenstein $R$-order. 
Suppose that $\LL$ is of uncountable CM representation type. Then $\KGdim \modsCML \neq 1$. 
\end{corollary}

Let $k$ be an uncountable algebraically closed field and $A$ a finite dimensional $k$-algebra. 
It is well known that every finitely generated indecomposable $A$-module admits an AR sequence.
Moreover, by the second Brauer--Thrall theorem, $A$ is either of finite representation type or of uncountable representation type. 
Therefore, as a direct consequence of Theorem \ref{C1}, we obtain the following stable analogue of the results of Herzog \cite{He97} and Krause \cite{Kr98}.

\begin{corollary}\label{C11}
Let $A$ be a finite dimensional $k$-algebra over an uncountable algebraically closed field $k$. 
Then $\KGdim \mod \underline{\mod A} \neq 1$. 
\end{corollary}


\ifx\undefined\bysame 
\newcommand{\bysame}{\leavevmode\hbox to3em{\hrulefill}\,} 
\fi

\end{document}